\documentclass[12pt]{amsart}
\usepackage{amsmath,amssymb}
\usepackage{amsthm}
\usepackage{hyperref}
\usepackage{xpatch}
\xpatchcmd{\qed}{\hfill}{}{}{}
\usepackage{mathrsfs}
\usepackage{theoremref}
\usepackage{enumerate}
\usepackage{vmargin}

\setmargrb{1in}{1in}{1in}{1in}
\usepackage{xspace}

\newcommand{\A}{\mathbb{A}}

\newcommand{\F}{\mathbb{F}}

\newcommand{\N}{\mathbb{N}}

\newcommand{\Z}{\mathbb{Z}}

\newcommand{\ccD}{\mathcal{D}}

\newcommand{\cN}{\mathcal{N}}
\newcommand{\cO}{\mathcal{O}}

\newcommand{\fg}{\mathfrak{g}}

\newcommand{\fL}{\mathfrak{L}}
\newcommand{\fm}{\mathfrak{m}}

\newcommand{\fn}{\mathfrak{n}}

\newcommand{\fz}{\mathfrak{z}}

\let\OldS\S
\renewcommand{\S}{\OldS\xspace}

\DeclareMathOperator{\ad}{ad}

\DeclareMathOperator{\Aut}{Aut}

\DeclareMathOperator{\Der}{Der}

\DeclareMathOperator{\GL}{GL}

\DeclareMathOperator{\id}{id}

\DeclareMathOperator{\Lie}{Lie}

\DeclareMathOperator{\modd}{mod}

\DeclareMathOperator{\del}{\partial}

\makeatletter

\usepackage{titlesec}

  \titleformat{\section}
  {\normalfont\bfseries\protect}
  {\thesection.}
  {0.5em}
  {}
  \titleformat{\subsection}[runin]
  {\normalfont\bfseries\protect}
  {\thesubsection.}
  {0.5em}
  {}

  \titleformat{\subsubsection}[runin]
  {\normalfont\bfseries\protect}
  {\thesubsubsection.}
  {0.5em}
  {}
  
\numberwithin{equation}{section}

\renewenvironment{proof}[1][\proofname]{%
   \par\pushQED{\qed}\normalfont%
   \topsep6\p@\@plus6\p@\relax
   \trivlist\item[\hskip\labelsep\bfseries#1\@addpunct{.}]%
   \ignorespaces
   }{%
   \popQED\endtrivlist\@endpefalse
}

\newtheorem{thm}{Theorem}[section]
\newtheorem{thm*}{Theorem}
\newtheorem{pro}{Proposition}[section]
\newtheorem{lem}{Lemma}[section]

\newtheorem{cor}{Corollary}[section]

\theoremstyle{definition}

\theoremstyle{definition}
\newtheorem{rmk}{Remark}[section]

\newtheorem{ex*}{Example}

\begin{document}
\title{The Nilpotent Variety of $W(1;n)_{p}$ is irreducible}
\author{Cong Chen}
\address{School of Mathematics, The University of Manchester, Oxford Road, M13 9PL, UK}
\email{cong.chen@manchester.ac.uk}

\begin{abstract}
In the late 1980s, Premet conjectured that the nilpotent variety of any finite dimensional restricted Lie algebra over an algebraically closed field of characteristic $p>0$ is irreducible. This conjecture remains open, but it is known to hold for a large class of simple restricted Lie algebras, e.g. for Lie algebras of connected reductive algebraic groups, and for Cartan series $W, S$ and $H$. In this paper, with the assumption that $p>3$, we confirm this conjecture for the minimal $p$-envelope $W(1;n)_p$ of the Zassenhaus algebra $W(1;n)$ for all $n\geq 2$. 
\end{abstract}

\maketitle

\section{Introduction}
Let $k$ be an algebraically closed field of characteristic $p>0$, and let $\fg$ be a finite dimensional restricted Lie algebra over $k$ with $[p]$-th power map $x\mapsto x^{[p]}$. The nilpotent variety $\cN(\fg)$ is the set of all $x\in\fg$ such that $x^{[p]^{N}}=0$ for $N\gg 0$. It is well known that $\cN(\fg)$ is Zariski closed in $\fg$, and it can be presented as a finite union $\cN(\fg)=Z_1\cup Z_2\cup \dots \cup Z_t$ of pairwise distinct irreducible components $Z_i$. In \cite{P90}, Premet conjectured that the variety $\cN(\fg)$ is irreducible, i.e. $t=1$ in this decomposition. The evidence to support this conjecture is that if $\fg$ is the Lie algebra of a connected reductive algebraic group $G'$, and $\fn$ is the set of nilpotent elements in a Borel subalgebra of $\fg$, then $\cN(\fg) =\{g.n \,|\, g\in G', n \in \fn\}$. Since $G'$ is connected and $\fn$ is irreducible, the variety $\cN(\fg)$ is irreducible \cite[p.~64]{J04}. Moreover, this conjecture holds for the Jacobson-Witt algebra $W(n; \underline{1})$ \cite{P91}, for the Special Lie algebras $S(n; \underline{1})$ \cite{WCL13} and for the Hamiltonian Lie algebras $H(2n; \underline{1})$ \cite{W14}. In this paper, we are interested in the minimal $p$-envelope of the Zassenhaus algebra.

Throughout this paper we assume that $k$ is an algebraically closed field of characteristic $p>3$, and $n\in \N$. The divided power algebra $\cO(1;n)$ has a $k$-basis $\{x^{(a)}\,|\, 0\leq a \leq p^{n}-1\}$, and the product in $\cO(1;n)$ is given by $x^{(a)}x^{(b)} = \binom {a + b}{a} x^{(a + b)}$ if $0\leq a+b \leq p^n-1$ and $0$ otherwise. We write $x^{(1)}$ as $x$. Note that $\cO(1;n)$ is a local algebra with the unique maximal ideal $\fm$ spanned by all $x^{(a)}$ such that $a\geq 1$. A system of divided powers is defined on $\fm$, $f\mapsto f^{(r)} \in \cO(1;n)$ where $r\geq0$; see \cite[Definition~2.1.1]{S04}. An automorphism $\Phi$ of $\cO(1;n)$ is called \textit{admissible} if $\Phi(f^{(r)})=\Phi(f)^{(r)}$ for all $f\in \fm$ and $r\geq 0$. Let $G$ denote the group of all admissible automorphisms of $\cO(1;n)$. It is well known that $G$ is a connected algebraic group of dimension $p^n-n$ \cite[Theorem~2]{W71}.

A derivation $\ccD$ of $\cO(1; n)$ is called \textit{special} if $\ccD(x^{(a)})=x^{(a-1)}\ccD(x)$ for $1\leq a \leq p^{n}-1$ and $0$ otherwise. The set of all special derivations of $\cO(1;n)$ forms a Lie subalgebra of $\Der(\cO(1;n))$ denoted $\fL=W(1;n)$ and called the \textit{Zassenhaus algebra}. It is well known that $\fL$ is a free $\cO(1;n)$-module of rank $1$ generated by the special derivation $\del$ such that $\del(x^{(a)})=x^{(a-1)}$ if $1\leq a \leq p^{n}-1$ and $0$ otherwise \cite[Ch. 4, Proposition~2.2(1)]{S88}. When $n=1$, $\fL$ coincides with the Witt algebra $W(1;1):=\Der (\cO(1;1))$, a simple and restricted Lie algebra. When $n\geq 2$, $\fL$ provides the first example of a simple, non-restricted Lie algebra \cite[Ch. 4, Theorem~2.4(1)]{S88}. From now on we always assume that $n\geq 2$. By \cite[Theorem~12.8]{R56}, any automorphism of $\fL$ is induced by a unique admissible automorphism of $\cO(1;n)$ so that $\Aut(\fL)\cong G$ as algebraic groups.

Let $\fL_p=W(1;n)_p$ denote the $p$-envelope of $\fL\cong \ad \fL$ in $\Der(\fL)$.  This semisimple restricted Lie algebra is referred to as the \textit{minimal $p$-envelope} of $\fL$. Recent studies have shown that the variety $\cN(\fL):=\cN(\fL_p)\cap \fL$ is reducible \cite{YS16}. So investigating the variety $\cN(\fL_p)$ becomes critical for verifying Premet's conjecture. Our main result is the following theorem.

\begin{thm}\thlabel{main}
The variety $\cN(\fL_p)$ coincides with the Zariski closure of
\[
\cN_\text{reg}:=G.(k^*\del+k\del^p+\dots+k\del^{p^{n-1}})
\] and hence is irreducible.
\end{thm}

Our paper is organized as follows. In Section $2$, we recall some basic results on $\fL$ and $\fL_p$. In Section $3$, we study some nilpotent elements of $\fL_p$ and then prove the main result. The proof is similar to Premet's proof for the Jacobson-Witt algebra $W(n;\underline{1})$\cite{P91}. It relies on the fact that the variety $\cN(\fL_p)$ is equidimensional of dimension $p^n-1$; see \cite[Theorem 4.2]{P03} and \cite[Theorem~7.6.3(2)]{S04}. Then we need to prove that $\cN_\text{sing}:=\cN(\fL_p)\setminus\cN_\text{reg}$ is Zariski closed of codimension $\geq n+1$ in $\fL_p$ by constructing an $(n+1)$-dimensional subspace $V$ in $\fL_p$ such that $V\cap \cN_\text{sing}=\{0\}$. The $(n+1)$-dimensional subspace used in $W(n;\underline{1})$ has no obvious analogue for $\fL_p$. Therefore, a new $V$ is constructed using the original definition of $\fL$ due to H.~Zassenhaus. In general, constructing analogues of $V$ for the minimal $p$-envelopes of $W(n;\underline{m})$, where $\underline{m}=(m_1, \dots, m_n)$ and $m_i>1$ for some $i$, would enable one to check Premet's conjecture for this class of restricted Lie algebras.

\section{Preliminaries}
\subsection{}\label{2.1}
Let $k$ be an algebraically closed field of characteristic $p>3$ and $n\in \N$. The divided power algebra $\cO(1;n)$ has a $k$-basis $\{x^{(a)}\,|\, 0\leq a \leq p^{n}-1\}$, and the product in $\cO(1;n)$ is given by $x^{(a)}x^{(b)} = \binom {a + b}{a} x^{(a + b)}$ if $0\leq a+b \leq p^n-1$ and $0$ otherwise. In the following, we write $x^{(1)}$ as $x$. It is straightforward to see that $\cO(1;n)$ is a local algebra with the unique maximal ideal $\fm$ spanned by all $x^{(a)}$ such that $a\geq 1$. A system of divided powers is defined on $\fm$, $f\mapsto f^{(r)} \in \cO(1;n)$ where $r\geq0$; see \cite[Definition~2.1.1]{S04}.

A derivation $\ccD$ of $\cO(1; n)$ is called \textit{special} if $\ccD(x^{(a)})=x^{(a-1)}\ccD(x)$ for $1\leq a \leq p^{n}-1$ and $0$ otherwise. The set of all special derivations of $\cO(1;n)$ forms a Lie subalgebra of $\Der(\cO(1;n))$ denoted $\fL=W(1;n)$ and called the \textit{Zassenhaus algebra}. When $n=1$, $\fL$ coincides with the Witt algebra $W(1;1):=\Der (\cO(1;1))$, a simple and restricted Lie algebra. When $n\geq 2$, $\fL$ provides the first example of a simple, non-restricted Lie algebra \cite[Ch. 4, Theorem~2.4(1)]{S88}. From now on we always assume that $n\geq 2$. 

The Zassenhaus algebra $\fL$ admits an $\cO(1;n)$-module structure via $(f\ccD)(x)=f\ccD(x)$ for all $f\in \cO(1;n)$ and $\ccD\in \fL$. Since each $\ccD\in \fL$ is uniquely determined by its effect on $x$, it is easy to see that $\fL$ is a free $\cO(1;n)$-module of rank $1$ generated by the special derivation $\del$ such that $\del(x^{(a)})=x^{(a-1)}$ if $1\leq a \leq p^{n}-1$ and $0$ otherwise \cite[Ch. 4, Proposition~2.2(1)]{S88}. Hence the Lie bracket in $\fL$ is given by $[x^{(i)}\del,x^{(j)}\del]=\big(\binom {i+j-1}{i}-\binom {i+j-1}{j}\big)x^{(i+j-1)}\del$ if $1\leq i+j \leq p^{n}$ and $0$ otherwise.

There is a $\Z$-grading on $\fL$, i.e. $\fL=\bigoplus_{i=-1}^{p^{n}-2} kd_{i}$ with $d_{i}:=x^{(i+1)}\partial$. Put $\fL_{(i)}:=\bigoplus_{j \geq i}^{p^{n}-2} kd_{j}$ for $-1\leq i\leq p^{n}-2$. Then this  $\Z$-grading induces a natural filtration
\[
\fL=\fL_{(-1)}\supset\fL_{(0)}\supset\fL_{(1)}\supset\dots \supset\fL_{(p^n-2)}\supset 0
\]
on $\fL$. It is known that
\begin{align}\label{eip}
d_{i}^p=
\begin{cases}
d_{i}, &\text{if $i=0$,}\\
d_{pi}, &\text{if $i=p^{t}-1$ for some $1\leq t\leq n-1$,}\\
0, &\text{otherwise}.
\end{cases}
\end{align}
In particular, $\fL_{(0)}$ is a restricted subalgebra of $\Der(\cO(1;n))$, $kd_0$ is a $1$-dimensional torus in $\fL_{(0)}$ and $\fL_{(1)}=\,\text{nil\,}(\fL_{(0)})$; see \cite[p.~3]{YS16}.

It is also useful to mention that the Zassenhaus algebra $\fL$ has another presentation. Let $q=p^n$ and let $\F_q \subset k$ be the set of all roots of $x^q-x=0$. This is a finite field of $q$ elements. Then $\fL$ has a $k$-basis $\{e_{\alpha}\,|\, \alpha\in \F_q\}$ with the Lie bracket given by $[e_\alpha, e_\beta]=(\beta-\alpha)e_{\alpha+\beta}$ \cite[Theorem~7.6.3(1)]{S04}. 

Let $\fL_p=W(1;n)_p$ denote the $p$-envelope of $\fL\cong \ad \fL$ in $\Der(\fL)$. This semisimple restricted Lie algebra is referred to as the \textit{minimal $p$-envelope} of $\fL$. By \cite[Theorems~7.1.2(1) and 7.2.2(1)]{S04}, we see that $\fL_p$ coincides with $\Der(\fL)=\fL +\sum_{i=1}^{n-1}k\partial^{p^{i}}$. Here we identify $\fL$ with $\ad \fL\subset \Der(\fL)$ and regard $\del^{p^{n}}$ as $0$. Then $\dim \fL_p=p^n+(n-1)$. 

Let $\cN$ denote the variety of nilpotent elements in $\fL_p$. It is well known that $\cN$ is Zariski closed in $\fL_p$. One should note that the maximal dimension of toral subalgebras in $\fL_p$ equals $n$ \cite[Theorem~7.6.3(2)]{S04}. Moreover, $\fL_p$ possesses a toral Cartan subalgebra; see \cite[p.~555]{P90}. Hence the set of all semisimple elements of $\fL_p$ is Zariski dense in $\fL_p$; see \cite[Theorem~2]{P87}. It follows from these facts, \cite[Corollary~2]{P90} and \cite[Theorem~4.2]{P03} that there exist nonzero homogeneous polynomial functions $\varphi_0, \dots, \varphi_{n-1}$ on $\fL_p$ such that $\cN$ coincides with the set of all common zeros of $\varphi_0, \dots, \varphi_{n-1}$. The variety $\cN$ is equidimensional of dimension $p^n-1$. Furthermore, any $\ccD \in \cN$ satisfies $\ccD^{p^{n}}=0$.

\subsection{}\label{2.2}
An automorphism $\Phi$ of $\cO(1;n)$ is called \textit{admissible} if $\Phi(f^{(r)})=\Phi(f)^{(r)}$ for all $f\in \fm$ and $r\geq 0$. Let $G$ denote the group of all admissible automorphisms of $\cO(1;n)$. It is well known that $G$ is a connected algebraic group, and each $\Phi \in G$ is uniquely determined by its effect on $x$. By \cite[Lemmas~8, 9 and 10]{W71}, an assignment $\Phi(x):=y=\sum_{i=1}^{p^{n}-1}\alpha_{i}x^{(i)}$ with $\alpha_{i}\in k$ such that $\alpha_{1}\neq 0$ and $\alpha_{p^{i}}=0$ for $1\leq i\leq n-1$ extends to an admissible automorphism of $\cO(1;n)$. Conversely, for any $y\in \fm$ as above, there is a unique $\Phi\in G$ such that $\Phi(x)=y$ \cite[Corollary~1]{W71}. Hence $\dim G=p^{n}-n$; see also \cite[Theorem~2]{W71}.

Any automorphism of the Zassenhaus algebra $\fL$ is induced by a unique admissible automorphism $\Phi$ of $\cO(1;n)$ via the rule $\ccD^{\Phi}=\Phi \ccD \Phi^{-1}$, where $\ccD\in \fL$ \cite[Theorem~12.8]{R56}. So from now on we shall identify $G$ with the automorphism group $\Aut(\fL)$. It is known that $G$ respects the natural filtration of $\fL$. In \cite{T98}, Tyurin stated explicitly that if $\Phi\in G$ is such that $\Phi(x)=y$, then $\Phi(g(x)\partial)=(y')^{-1}g(y)\partial$ for any $g(x)\in \cO(1;n)$. Extend this by defining $\Phi(\del^{p^{i}})=\Phi(\del)^{p^{i}}$ for $1\leq i\leq n-1$ one gets an automorphism of $\fL_p$. 

It follows from the above description of $G$ that $\Lie(G)\subseteq \fL_{(0)}$. More precisely, 
\begin{lem}\thlabel{LieG}
The set $\{d_i\,|\,\text{$0 \leq i \leq p^n-2$ and $i\neq p^t-1$ for $1\leq t\leq n-1$}\}$ forms a $k$-basis of $\Lie(G)$.
\end{lem} 
\begin{proof}
Let $\psi :\A^{1}\rightarrow G$ be the map defined by $t\mapsto (x\mapsto x+tx^{(i+1)})$, where $0 \leq i \leq p^n-2$ and $i\neq p^t-1$ for $1\leq t\leq n-1$. It is easy to check that $\psi$ is a morphism of algebraic varieties. Then the differential $d_0\psi$ of $\psi$ at $0$ is the map $d_0\psi: k \rightarrow \Lie(G)$. So $d_0\psi(k)\subseteq \Lie(G)$.

Let us compute $d_0\psi(k)$. The morphism $\psi$ sends $\A^{1}$ to the set of admissible automorphisms $\{\Phi_t\,|\, t \in \A^{1}\}$, where $\Phi_t(x)=x+tx^{(i+1)}$. Since $\Phi_t$ is uniquely determined by its effect on $x$ and ``admissible'' is equivalent to the condition that $\Phi_t(x^{(p^{j})})=\Phi_t(x)^{(p^{j})}$ for $1\leq j\leq n-1$; see \cite[Lemma~8]{W71}. Then by \cite[Definition~2.1.1]{S04} we have that 
\begin{align*}
\Phi_t(x^{(p^{j})})&=(x+tx^{(i+1)})^{(p^{j})}=x^{(p^{j})}+tx^{(p^{j}-1)}x^{(i+1)}+\text{terms of higher degree in $t$}.
\end{align*} 
Passing to $d_0\psi(t)$ we get
\begin{align*}
x&\mapsto x^{(i+1)},\\
x^{(p^{j})}&\mapsto x^{(p^{j}-1)}x^{(i+1)}.
\end{align*}
These results are the same as $d_i=x^{(i+1)}\del$ acting on $x$ and $x^{(p^{j})}$, respectively. Hence $d_i\in d_0\psi(t)\subseteq d_0\psi(k)$. Note that $\{d_i\,|\,\text{$0 \leq i \leq p^n-2$ and $i\neq p^t-1$ for $1\leq t\leq n-1$}\}$ is a set of $p^n-n$ linearly independent vectors. Since $\dim\Lie(G)=\dim G=p^n-n$, they form a basis of $\Lie(G)$. This completes the proof. 
\end{proof}

\section{The variety \texorpdfstring{$\cN$}{N}}
\subsection{} In~\S\ref{2.1}, we observed that any elements of $\fL_{(1)}$ are nilpotent, but they do not tell us much information about $\cN$. The interesting nilpotent elements are contained in the complement of $\fL_{(1)}$ in $\cN$, denoted $\cN\setminus\fL_{(1)}$. They are of the form $\sum_{i=0}^{n-1}\alpha_{i}\partial^{{p^{i}}}+f(x)\partial$ for some $f(x)\in \fm$ and $\alpha_i\in k$ with at least one $\alpha_{i}\neq 0$. In this subsection, we study elements of this form.

\begin{lem}\thlabel{l:l13}
Let $\ccD=\partial^{p^{n-1}}+\sum_{i=0}^{n-2}\beta_{i}\partial^{p^{i}}+g(x)\partial$ be an element of $\fL_p$, where $\beta_{i} \in k$ and $g(x)\in \fm$. Then $\ccD$ is conjugate under $G$ to
\begin{align*}
\partial^{p^{n-1}}+\sum_{i=0}^{n-2}\beta_{i}\partial^{p^{i}}+x^{(p^n-p^{n-1})}h(x)\partial
\end{align*}
for some $h(x)=\sum_{i=0}^{p^{n-1}-1}\mu_{i}x^{(i)}$ with $\mu_{i}\in k$.
\end{lem}
\begin{proof}
Take $\ccD=\partial^{p^{n-1}}+\sum_{i=0}^{n-2}\beta_{i}\partial^{p^{i}}+g(x)\partial$ as in the lemma. By  the proof of \cite[Theorem 1]{T98}, if $\Phi(x)=y$ is any admissible automorphism of $\cO(1;n)$ with identical linear part, then 
\begin{align*}
\Phi(\ccD)&=\partial^{p^{n-1}}+\sum_{i=1}^{n-2}\beta_{i}\partial^{p^{i}}+(y')^{-1}\big(\beta_{0}+\Phi(g(x))-\sum_{i=1}^{n-2}\beta_{i}\partial^{p^{i}}y-\partial^{p^{n-1}}y\big)\partial.\\
\intertext{If $g(x)\partial\equiv \gamma_{1}x\partial\,(\modd \fL_{(1)})$ for some $\gamma_1\in k$ and $\Phi(x)=y=x+\gamma_{1}x^{(p^{n-1}+1)}$, then we can show that}
\Phi(\ccD) &\equiv \partial^{p^{n-1}}+\sum_{i=1}^{n-2}\beta_{i}\partial^{p^{i}}+\beta_{0}\del\,(\modd \fL_{(1)}). 
\end{align*}
For $\gamma_1=0$, i.e. $\Phi$ is the identity automorphism, the result is clear. For $\gamma_1\in k^{*}$, let us show this congruence by proving that
\begin{align}\label{e:q1}
(y')^{-1}\big(\beta_{0}+\Phi(g(x))-\sum_{i=1}^{n-2}\beta_{i}\partial^{p^{i}}y-\partial^{p^{n-1}}y\big)\del- \beta_{0}\del \in \fL_{(1)}.
\end{align}
Note that $y'=1+\gamma_1x^{(p^{n-1} )}$ which is invertible in $\cO(1; n)$. Since $\fL_{(1)}$ is invariant under multiplication of invertible elements of $\cO(1;n)$, we can multiply both sides of \eqref{e:q1} by $y'$ and show that
\begin{align*}
\big(\beta_{0}+\Phi(g(x))-\sum_{i=1}^{n-2}\beta_{i}\partial^{p^{i}}y-\partial^{p^{n-1}}y\big)\del- \beta_{0} y'\del\in\fL_{(1)}.
\end{align*}
Since $g(x)\partial\equiv \gamma_{1}x\partial\,(\modd \fL_{(1)})$ and $\Phi$ preserves the natural filtration of $\fL$, in particular, it preserves $\fL_{(1)}$, hence 
\begin{align*}
&\big(\beta_{0}+\Phi(g(x))-\sum_{i=1}^{n-2}\beta_{i}\partial^{p^{i}}y-\partial^{p^{n-1}}y\big)\del- \beta_{0} y'\del\\
\equiv&\big(\beta_{0}+\gamma_{1}(x+\gamma_1 x^{(p^{n-1}+1)})-\sum_{i=1}^{n-2}\beta_{i}\gamma_{1}x^{(p^{n-1}-p^{i}+1)}-\gamma_1 x\big)\del-\beta_0(1+\gamma_1x^{(p^{n-1} )})\del\\
\equiv & 0\,(\modd \fL_{(1)}).
\end{align*}
Therefore, $\ccD$ is conjugate to $\del^{p^{n-1}}+\sum_{i=1}^{n-2}\beta_{i}\partial^{p^{i}}+\beta_{0}\partial\, (\modd \fL_{(1)})$. If $\ccD \equiv \partial^{p^{n-1}}+\sum_{i=1}^{n-2}\beta_{i}\partial^{p^{i}}+\beta_{0}\partial+\gamma_2x^{(2)}\del\,(\modd \fL_{(2)})$ for some $\gamma_2 \in k$, i.e. $\ccD =\partial^{p^{n-1}}+\sum_{i=1}^{n-2}\beta_{i}\partial^{p^{i}}+\beta_{0}\partial+g_2(x)\del$ with $g_2(x)\del\equiv\gamma_2x^{(2)}\del\,(\modd\fL_{(2)})$, then repeat the above process by applying the automorphism $\Phi_2(x)=y_2=x+\gamma_2 x^{(p^{n-1}+2)}$ to $\ccD$ we have that
\begin{align*}
\Phi_2(\ccD)&=\partial^{p^{n-1}}+\sum_{i=1}^{n-2}\beta_{i}\partial^{p^{i}}+(y_{2}')^{-1}\big(\beta_{0}+\Phi_2(g_2(x))-\sum_{i=1}^{n-2}\beta_{i}\partial^{p^{i}}y_2-\partial^{p^{n-1}}y_2\big)\partial.
\intertext{Then we can show that}
\Phi_2(\ccD)&\equiv\partial^{p^{n-1}}+\sum_{i=1}^{n-2}\beta_{i}\partial^{p^{i}}+\beta_0\del\,(\modd\fL_{(2)}).
\end{align*}
For $\gamma_2=0$, the result is clear. For $\gamma_2 \in k^{*}$, let us prove that
\begin{align}
&(y_2')^{-1}\big(\beta_{0}+\Phi_2(g_2(x))-\sum_{i=1}^{n-2}\beta_{i}\partial^{p^{i}}y_2-\partial^{p^{n-1}}y_2\big)\partial-\beta_{0}\del\in \fL_{(2)}.\label{e:q3}
\end{align}
By the same reasons as before, we can multiply both sides of \eqref{e:q3} by $y_2'$ and show that 
\begin{align*}
(\beta_{0}+\Phi_2(g_2(x))-\sum_{i=1}^{n-2}\beta_{i}\partial^{p^{i}}y_2-\partial^{p^{n-1}}y_2)\del - \beta_{0}y_2'\del \in\fL_{(2)}.
\end{align*}
Indeed, since $\fL_{(2)}$ is invariant under $\Phi_2$ we have that
\begin{align*}
&(\beta_{0}+\Phi_2(g_2(x))-\sum_{i=1}^{n-2}\beta_{i}\partial^{p^{i}}y_2-\partial^{p^{n-1}}y_2)\del - \beta_{0}y_2'\del\\
\equiv &\big(\beta_{0}+\gamma_{2}(x+\gamma_{2}x^{(p^{n-1}+2)})^{(2)}-\sum_{i=1}^{n-2}\beta_{i}\gamma_{2}x^{(p^{n-1}-p^{i}+2)}-\gamma_{2}x^{(2)}\big)\del-\beta_{0} (1+\gamma_{2}x^{(p^{n-1}+1)})\del\\
\equiv&0 \,(\modd\fL_{(2)}).
\end{align*}
Hence $\ccD$ is conjugate to $\partial^{p^{n-1}}+\sum_{i=1}^{n-2}\beta_{i}\partial^{p^{i}}+\beta_{0}\partial\, (\modd \fL_{(2)})$. Then 
\begin{align*}
\ccD\equiv\partial^{p^{n-1}}+\sum_{i=1}^{n-2}\beta_{i}\partial^{p^{i}}+\beta_{0}\del+\gamma_{3}x^{(3)}\partial\,(\modd \fL_{(3)})
\end{align*}
for some $\gamma_3\in k$. Apply the automorphism $\Phi_3(x)=y_3=x+\gamma_{3}x^{(p^{n-1}+3)}$ to $\ccD$ we can show that $\ccD$ is conjugate to $\partial^{p^{n-1}}+\sum_{i=1}^{n-2}\beta_{i}\partial^{p^{i}}+\beta_{0}\partial\,(\modd \fL_{(3)})$. Continue doing this until we get $\ccD$ is conjugate to $\partial^{p^{n-1}}+\sum_{i=1}^{n-2}\beta_{i}\partial^{p^{i}}+\beta_{0}\partial\, (\modd \fL_{(p^{n}-p^{n-1}-1)})$. Then
\begin{align*}
\ccD\equiv\partial^{p^{n-1}}+\sum_{i=1}^{n-2}\beta_{i}\partial^{p^{i}}+\beta_{0}\partial+\gamma_{p^{n}-p^{n-1}}x^{(p^{n}-p^{n-1})}\del\, (\modd \fL_{(p^{n}-p^{n-1})})
\end{align*}
for some $\gamma_{p^{n}-p^{n-1}} \in k$. Next, we were supposed to apply the automorphism $\Phi_{p^{n}-p^{n-1}}(x)=x+\gamma_{p^{n}-p^{n-1}}x^{(p^n)}$ to $\ccD$. But since $x^{(j)}=0$  for $j\geq p^n$ in $\cO(1; n)$, the automorphism $\Phi_{p^{n}-p^{n-1}}$ is the identity automorphism and we stop here. Therefore, $\ccD$ is conjugate under $G$ to 
\begin{align*}
\partial^{p^{n-1}}+\sum_{i=0}^{n-2}\beta_{i}\partial^{p^{i}}+x^{(p^{n}-p^{n-1})}h(x)\partial
\end{align*}
for some $h(x)=\sum_{i=0}^{p^{n-1}-1}\mu_{i}x^{(i)}$ with $\mu_{i}\in k$. This completes the proof.
\end{proof}

\begin{cor}\thlabel{c:c1}
Let $\ccD=\partial^{p^{m}}+\sum_{i=0}^{m-1}\beta_{i}\partial^{p^{i}}+g(x)\partial$ be an element of $\fL_p$, where $1\leq m\leq n-2, \beta_{i}\in k$ and $g(x)\in\fm$. Then $\ccD$ is conjugate under $G$ to 
\begin{align*}
\partial^{p^{m}}+\sum_{i=0}^{m-1}\beta_{i}\partial^{p^{i}}+x^{(p^n-p^{m})}h(x)\partial
\end{align*}
for some $h(x)=\sum_{i=0}^{p^{m}-1}\mu_{i}x^{(i)}$ with $\mu_{i}\in k$.
\end{cor}
\begin{proof}
If $g(x)\del\equiv \gamma_{i}x^{(i)}\del \,(\modd \fL_{(i)})$ for some $\gamma_i \in k$, then the automorphism $\Phi(x)=x+\gamma_{i}x^{(p^{m}+i)}$ reduces $\ccD$ to the form $\ccD\equiv \del^{p^{m}}+\sum_{i=1}^{m-1}\beta_{i}\del^{p^{i}}+\beta_{0}\del\, (\modd \fL_{(i)})$. Continue doing this until we get $\ccD\equiv \del^{p^{m}}+\sum_{i=1}^{m-1}\beta_{i}\del^{p^{i}}+\beta_{0}\del\, (\modd \fL_{(p^{n}-p^m-1)})$. This completes the proof.
\end{proof}

\begin{lem}\thlabel{p:p1}
Let $\ccD=\partial^{p^{n-1}}+\sum_{i=0}^{n-2}\beta_{i}\partial^{p^{i}}+x^{(p^{n}-p^{n-1})}\sum_{i=0}^{p^{n-1}-1}\mu_{i}x^{(i)}\partial$ be a nilpotent element of $\fL_p$.
\begin{enumerate}[\upshape(i)]
\item If $\beta_{i}=0$ for all $i$, then $\mu_{0}=\mu_{1}=0$ and $\ccD^{p^{n-1}}\in \fL_{(1)}$.
\item 
\begin{enumerate}[\upshape(a)]
\item If $j\geq 0$ is the smallest index such that $\beta_{j}\neq 0$, then $\mu_{0}=0$ and $\ccD^{p^{n-1-j}}$ is conjugate under $G$ to 
\[
\del^{p^{n-1}}+x^{(p^{n}-p^{n-1})}\sum_{i=2}^{p^{n-1}-1}\nu_{i}x^{(i)}\del
\]
for some $\nu_{i}\in k$. Hence $\ccD^{p^{n-1}}\in \fL_{(1)}$ for all $j\geq 1$.
\item In particular, if $\beta_{0}\neq 0$, then $\ccD^{p^{n-1}}$ is conjugate under $G$ to $\del^{p^{n-1}}$. Hence $\ccD=\del^{p^{n-1}}+\sum_{i=0}^{n-2}\gamma_{i}\del^{p^{i}}$ for some $\gamma_{i}\in k$ with $\gamma_{0}\neq 0$.
\end{enumerate}
\end{enumerate}
\end{lem}
\begin{proof}
Let $\ccD=\partial^{p^{n-1}}+\sum_{i=0}^{n-2}\beta_{i}\partial^{p^{i}}+x^{(p^{n}-p^{n-1})}\sum_{i=0}^{p^{n-1}-1}\mu_{i}x^{(i)}\partial$ be a nilpotent element of $\fL_p$. Then $\ccD^{p^{n}}=0$; see \S\ref{2.1}. Let us first calculate $\ccD^{p}$. Recall Jacobson's formula,
\begin{align}\label{J1}
(\ccD_1+\ccD_2)^{p}=\ccD_1^{p}+\ccD_2^{p}+\sum_{i=1}^{p-1}s_{i}(\ccD_1, \ccD_2)
\end{align}
for all $\ccD_1, \ccD_2\in \fL_p$, and $s_{i}(\ccD_1, \ccD_2)$ can be computed by the formula
\begin{align*}
\ad(t\ccD_1+\ccD_2)^{p-1}(\ccD_1)=\sum_{i=1}^{p-1}is_{i}(\ccD_1, \ccD_2)t^{i-1},
\end{align*}
where $t$ is a parameter. Set $\ccD_1=x^{(p^{n}-p^{n-1})}\sum_{i=0}^{p^{n-1}-1}\mu_{i}x^{(i)}\partial$ and $\ccD_2=\partial^{p^{n-1}}+\sum_{i=0}^{n-2}\beta_{i}\partial^{p^{i}}$. Then $\ccD_1^{p}=0$ by \eqref{eip} and $\ccD_2^{p}=\sum_{i=0}^{n-2}\beta_{i}^{p}\partial^{p^{i+1}}$. By the natural filtration of $\fL$, we have that for any $1\leq s\leq p-2$, 
\begin{align*}
[\ccD_1, (\ad \ccD_2)^{s}(\ccD_1)]&\in [\fL_{(p^{n}-p^{n-1}-1)}, \fL_{(p^{n}-(s+1)p^{n-1}-1)}]\\
&\subseteq[\fL_{(p^{n}-p^{n-1}-1)}, \fL_{(p^{n-1}-1)}]\\
&\subseteq \fL_{(p^{n}-2)}=\text{span$\{x^{(p^{n}-1)}\partial\}$}.
\end{align*}
This last term will appear if and only if $s=p-2$. So
\begin{align}\label{cal}
\ccD^{p}&=\sum_{i=0}^{n-2}\beta_{i}^{p}\partial^{p^{i+1}}+(\ad \ccD_2)^{p-1}(\ccD_1)+\mu_{(1)} x^{(p^{n}-1)}\partial
\end{align}
for some $\mu_{(1)}\in k$.

(i) If $\beta_{i}=0$ for all $i$, then $\ccD^{p}=(\ad \del^{p^{n-1}})^{p-1}(\ccD_1)+\mu_{(1)} x^{(p^{n}-1)}\partial$. Since $\partial^{p^{n-1}}$ is a derivation of $\fL$ and $\partial^{p^{n-1}}(\sum_{i=0}^{p^{n-1}-1}\mu_{i}x^{(i)})=0$, we have that
\begin{align*}
\ccD^{p}
&=\sum_{i=0}^{p^{n-1}-1}\mu_{i}x^{(i)}\del+\mu_{(1)} x^{(p^{n}-1)}\partial.
\intertext{If $\mu_{0}\neq 0$, then $\ccD^{p}\equiv \mu_{0}\partial \,(\modd \fL_{(0)})$. By Jacobson's formula,} 
\ccD^{p^{n}}&\equiv \mu_{0}^{p^{n-1}}\partial^{p^{n-1}}+\sum_{i=0}^{n-2} \mu_{i}'\del^{p^{i}}\,(\modd \fL_{(0)}) 
\end{align*}
for some $\mu_{i}'\in k$. As $\mu_0\neq 0$, this implies that $\ccD^{p^{n}}\not \equiv 0\,(\modd \fL_{(0)})$ and so is not equal to $0$. This is a contradiction. Hence $\mu_{0}=0$. Similarly, if $\mu_{1}\neq 0$ then $\ccD^{p}\equiv \mu_{1}x\partial \,(\modd \fL_{(1)})$. But $\ccD^{p^{n}}\equiv \mu_{1}^{p^{n-1}}x\partial\not \equiv 0\,(\modd \fL_{(1)})$, a contradiction. Thus $\mu_{1}=0$. Therefore, $\ccD^{p}=\sum_{i=2}^{p^{n-1}-1}\mu_{i}x^{(i)}\del+\mu_{(1)} x^{(p^{n}-1)}\partial$, which is an element of $\fL_{(1)}$. Since $\fL_{(1)}$ is restricted we have that $\ccD^{p^{n-1}}\in \fL_{(1)}$. This proves statement (i).

(ii)(a) Let $j\geq 0$ be the smallest index such that $\beta_{j}\neq 0$, and let $l$ be the largest index such that $\beta_{l}\neq 0$, i.e. $0\leq j\leq l\leq n-2$. Let us consider the special case $j=l$, i.e.
\begin{align*}
\ccD=\partial^{p^{n-1}}+\beta_{j}\partial^{p^{j}}+x^{(p^{n}-p^{n-1})}\sum_{i=0}^{p^{n-1}-1}\mu_{i}x^{(i)}\partial.
\end{align*}
We prove by induction that for any $1\leq r\leq n-1-j$, $\ccD^{p^{r}}$ is conjugate under $G$ to 
\begin{align*}
\partial^{p^{j+r}}+\beta_{0,(1)}^{p^{r-1}}\del^{p^{r-1}}+x^{(p^n-p^{j+r})}\sum_{i=0}^{p^{j+r}-1}\mu_{i,(r)}x^{(i)}\del
\end{align*}
for some $\beta_{0,(1)}\in k^{*}\mu_0$ and $\mu_{i,(r)} \in k$. For $r=1$, the previous calculation \eqref{cal} gives
\begin{align*}
\ccD^{p}=\beta_{j}^{p}\partial^{p^{j+1}}+\ad \big(\partial^{p^{n-1}}+\beta_{j}\partial^{p^{j}}\big)^{p-1}\big(x^{(p^{n}-p^{n-1})}\sum_{i=0}^{p^{n-1}-1}\mu_{i}x^{(i)}\partial\big)+\mu_{(1)} x^{(p^{n}-1)}\partial.
\end{align*}
Note that
\begin{align*}
&\ad \big(\partial^{p^{n-1}}+\beta_{j}\partial^{p^{j}}\big)^{p-1}\big(x^{(p^{n}-p^{n-1})}\sum_{i=0}^{p^{n-1}-1}\mu_{i}x^{(i)}\partial\big)\\
=&\ad\big(\sum_{m=0}^{p-1}(-1)^{m}\beta_{j}^{p-1-m}\partial^{mp^{n-1}+(p-1-m)p^{j}}\big)\big(\sum_{i=0}^{p^{n-1}-1}\mu_{i}x^{(p^{n}-p^{n-1}+i)}\partial\big)\\
=&\sum_{i=0}^{p^{n-1}-1}\mu_{i}x^{(i)}\partial-\beta_{j}\sum_{i=0}^{p^{n-1}-1}\mu_{i}x^{(p^{n-1}-p^{j}+i)}\partial+\dots+\beta_{j}^{p-1}\sum_{i=0}^{p^{n-1}-1}\mu_{i}x^{(p^{n}-p^{n-1}-(p-1)p^{j}+i)}\partial.
\end{align*}
The above result can be rewritten as $\mu_{0}\del+g(x)\del$ for some $g(x)\in \fm$. Hence
\[
\ccD^{p}=\beta_{j}^{p}\partial^{p^{j+1}}+\mu_{0}\del+g(x)\del+\mu_{(1)} x^{(p^{n}-1)}\partial.
\]
Then the automorphism $\Phi(x)=\alpha x$ with $\alpha^{p^{j+1}}=\beta_j^{p}$ reduces $\ccD^{p}$ to the form 
\[
\ccD^{p}=\partial^{p^{j+1}}+\beta_{0,(1)}\del+f_{1}(x)\del,
\]
where $\beta_{0,(1)}\in k^{*}\mu_{0}$ and $f_{1}(x)\in \fm$. It follows from \thref{l:l13} and \thref{c:c1} that $\ccD^{p}$ is conjugate under $G$ to 
\begin{align*}
\partial^{p^{j+1}}+\beta_{0,(1)}\del+x^{(p^n-p^{j+1})}\sum_{i=0}^{p^{j+1}-1}\mu_{i,(1)}x^{(i)}\del 
\end{align*}
for some $\mu_{i,(1)}\in k$. Thus the result is true for $r=1$. Suppose the result is true for $r=k-1$, i.e. $\ccD^{p^{k-1}}$ is conjugate under $G$ to 
\[
\partial^{p^{j+k-1}}+\beta_{0,(1)}^{p^{k-2}}\del^{p^{k-2}}+x^{(p^n-p^{j+k-1})}\sum_{i=0}^{p^{j+k-1}-1}\mu_{i,(k-1)}x^{(i)}\del
\] 
for some $\mu_{i,(k-1)}\in k$. Let us calculate $\ccD^{p^{k}}$. Set $\ccD_1=x^{(p^n-p^{j+k-1})}\sum_{i=0}^{p^{j+k-1}-1}\mu_{i,(k-1)}x^{(i)}\del$ and $\ccD_2=\partial^{p^{j+k-1}}+\beta_{0,(1)}^{p^{k-2}}\del^{p^{k-2}}$ in the Jacobson's formula~\eqref{J1}. Then $\ccD_1^p\in\fL_{(1)}$ and $\ccD_2^p=\partial^{p^{j+k}}+\beta_{0,(1)}^{p^{k-1}}\del^{p^{k-1}}$. By the natural filtration of $\fL$, we have that 
\begin{align*}
(\ad \ccD_2)^{p-1}(\ccD_1)\in \fL_{(p^n-p^{j+k}-1)}\subseteq\fL_{(1)}.
\end{align*}
Similarly, for any $1\leq s\leq p-2$, 
\begin{align*}
[\ccD_1, (\ad \ccD_2)^s(\ccD_1)]&\in [\fL_{(p^n-p^{j+k-1}-1)}, \fL_{(p^n-(s+1)p^{j+k-1}-1)}]\\
&\subseteq[\fL_{(1)}, \fL_{(p^n-(s+1)p^{j+k-1}-1)}]\\
&\subseteq\fL_{(p^n-(s+1)p^{j+k-1})},\\
&\subseteq \fL_{(p^n-(p-1)p^{j+k-1})},\\
&\subseteq \fL_{(p^n-(p-1)p^{n-2})}\,\text{(since $1\leq j+k-1\leq n-2$)}\\
&\subseteq \fL_{(1)}.
\end{align*}
Hence $\ccD^{p^{k}}=\partial^{p^{j+k}}+\beta_{0,(1)}^{p^{k-1}}\del^{p^{k-1}}+f_{k}(x)\del$ for some $f_{k}(x)\in \fm$. By \thref{l:l13} and \thref{c:c1}, $\ccD^{p^{k}}$ is conjugate under $G$ to 
\[
\partial^{p^{j+k}}+\beta_{0,(1)}^{p^{k-1}}\del^{p^{k-1}}+x^{(p^n-p^{j+k})}\sum_{i=0}^{p^{j+k}-1}\mu_{{i}, (k)}x^{(i)}\del
\]for some $\mu_{{i}, (k)}\in k$, i.e. the result is true for $r=k$. Therefore, we proved by induction that for any $1\leq r\leq n-1-j$, $\ccD^{p^{r}}$ is conjugate under $G$ to 
\begin{align*}
\partial^{p^{j+r}}+\beta_{0,(1)}^{p^{r-1}}\del^{p^{r-1}}+x^{(p^n-p^{j+r})}\sum_{i=0}^{p^{j+r}-1}\mu_{i,(r)}x^{(i)}\del
\end{align*}
for some $\beta_{0,(1)}\in k^{*}\mu_{0}$ and $\mu_{i,(r)} \in k$. In particular, $\ccD^{p^{n-1-j}}$ is conjugate under $G$ to 
\begin{align}
\partial^{p^{n-1}}+\beta_{0,(1)}^{p^{n-2-j}}\del^{p^{n-2-j}}+x^{(p^n-p^{n-1})}\sum_{i=0}^{p^{n-1}-1}\mu_{i,(n-1-j)}x^{(i)}\del\label{ee1}.
\end{align}
By Jacobson's formula, 
\begin{align}
\ccD^{p^{n-j}}=\beta_{0,(1)}^{p^{n-1-j}}\del^{p^{n-1-j}}+\sum_{i=0}^{p^{n-1}-1}\mu_{i,(n-1-j)}x^{(i)}\del+f_{n-j}(x) \del+\mu_{(n-j)}x^{(p^{n}-1)}\del\,\label{ee2}
\end{align} for some $f_{n-j}(x)\del \in \fL_{(1)}$ and $\mu_{(n-j)}\in k$. Then
\begin{align*}
\ccD^{p^{n}}\equiv \beta_{0,(1)}^{p^{n-1}}\del^{p^{n-1}}+\mu_{0,(n-1-j)}^{p^{j}}\del^{p^{j}}+\sum_{i=0}^{j-1}\mu_i'\del^{p^{i}}\,(\modd \fL_{(0)})
\end{align*}
for some $\mu_i' \in k$. But $\ccD^{p^{n}}=0$, this implies that $\beta_{0,(1)}=0$ and so $\mu_{0}=0$. We must also have that $\mu_{0,(n-1-j)}=0$ and $\mu_i'=0$ for all $i$. Substitute these into \eqref{ee2}, we get 
\begin{align*}
\ccD^{p^{n-j}}&=\sum_{i=1}^{p^{n-1}-1}\mu_{i,(n-1-j)}x^{(i)}\del +f_{n-j}(x) \del+\mu_{(n-j)}x^{(p^{n}-1)}\del\\
&\equiv \mu_{1,(n-1-j)}x\del \,(\modd \fL_{(1)}).
\end{align*}
Then one can show similarly that $\mu_{1,(n-1-j)}=0$. Hence $\ccD^{p^{n-1-j}}$ \eqref{ee1} is conjugate under $G$ to 
\begin{align*}
\partial^{p^{n-1}}+x^{(p^n-p^{n-1})}\sum_{i=2}^{p^{n-1}-1}\mu_{i,(n-1-j)}x^{(i)}\del.
\end{align*}
If $j<l$, i.e. $\ccD=\partial^{p^{n-1}}+\sum_{i=j}^{l}\beta_{i}\partial^{p^{i}}+x^{(p^{n}-p^{n-1})}\sum_{i=0}^{p^{n-1}-1}\mu_{i}x^{(i)}\partial$, then one can show similarly that $\ccD^{p^{n-1-j}}$ is conjugate under $G$ to 
\[
\partial^{p^{n-1}}+\lambda\del^{p^{n-2-j}}+\sum_{i=0}^{n-3-j}\lambda_i\del^{p^{i}}+x^{(p^n-p^{n-1})}\sum_{i=0}^{p^{n-1}-1}\nu_ix^{(i)}\del
\] for some $\lambda\in k^{*}\mu_0$ and $\lambda_i, \nu_i \in k$. Then by the same arguments as above, one can show that $\mu_{0}=0$ and $\ccD^{p^{n-1-j}}$ is conjugate under $G$ to 
\begin{align*}
&\partial^{p^{n-1}}+x^{(p^{n}-p^{n-1})}\sum_{i=2}^{p^{n-1}-1}\nu_ix^{(i)}\partial.
\end{align*} 
Suppose now $j\geq 1$. By Jacobson's formula,  
\begin{align*}
\ccD^{p^{n-j}}=&\sum_{i=2}^{p^{n-1}-1}\nu_{i}x^{(i)}\partial+\mu_{(n-j)}x^{(p^n-1)}\del
\end{align*}
for some $\mu_{(n-j)}\in k$. This is an element of $\fL_{(1)}$. Since $G$ preserves $\fL_{(1)}$ we have that $\ccD^{p^{n-j}}\in \fL_{(1)}$. As $\fL_{(1)}$ is restricted, we have that $\ccD^{p^{n-1}}\in \fL_{(1)}$. This proves statement (ii)(a).

(b) If $\beta_{0}\neq 0$, then (ii)(a) implies that $\ccD^{p^{n-1}}$ is conjugate under $G$ to 
\[
\partial^{p^{n-1}}+x^{(p^{n}-p^{n-1})}\sum_{i=2}^{p^{n-1}-1}\nu_{i}x^{(i)}\del
\]
for some $\nu_{i}\in k$. If $q$ is the smallest index such that $\nu_{q}\neq 0$, then 
\begin{align*}
\ccD^{p^{n}}=&\sum_{i=q}^{p^{n-1}-1}\nu_{i}x^{(i)}\del+\mu_{(n)}x^{(p^{n}-1)}\del
\end{align*}
for some $\mu_{(n)}\in k$. As $\nu_{q}\neq 0$, this implies that $\ccD^{p^{n}}\neq 0$, a contradiction. Hence $\nu_{i}=0$ for all $i$. Therefore, we are interested in the set 
\begin{align*}
\mathcal{S}:=\{\ccD\in \big(\del^{p^{n-1}}+\sum_{i=1}^{n-2}k\del^{p^{i}}+\fL\big)\cap \cN\,|\, \text{$\ccD^{p^{n-1}}$ is conjugate under $G$ to $\partial^{p^{n-1}}$}\}.
\end{align*}
Since $[\ccD, \ccD^{p^{n-1}}]=0$, the above set $\mathcal{S}$ is a subset of the centraliser $\fz_{\fL_{p}}(\del^{p^{n-1}})$ of $\del^{p^{n-1}}$ in $\fL_p$. It is easy to verify that $\fz_{\fL_{p}}(\del^{p^{n-1}})$ is spanned by $\del^{p^{n-1}}$ and $W(1,n-1)_{p}$. Since $W(1,n-1)_{p}$ is a restricted Lie subalgebra of $\fL_{p}$, we may regard the automorphism group of $W(1,n-1)_{p}$ as a subgroup of $G$. Let $\ccD=\del^{p^{n-1}}+\sum_{i=1}^{n-2}\gamma_{i}\del^{p^{i}}+n$ be an element of $\fz_{\fL_{p}}(\del^{p^{n-1}})$, where $\gamma_{i}\in k$ and $n\in W(1,n-1)$. If $n=0$, then $\ccD^{p^{n-1}}=0$ which is not conjugate to $\del^{p^{n-1}}$. So $n\neq 0$. If $n\not\in W(1, n-1)_{(0)}$, then $n=\gamma_{0}\del$ for some $\gamma_{0}\neq 0$. It is easy to see that $\ccD^{p^{n-1}}$ is conjugate under $G$ to $\del^{p^{n-1}}$. If $n \in W(1, n-1)_{(0)}$, then $n=\sum_{i=1}^{p^{n-1}-1}\lambda_{i}x^{(i)}\del$ with $\lambda_{i}\neq 0$ for some $i$. It follows from \thref{l:l13} that $\ccD$ is conjugate under $G$ to
\begin{align*}
\del^{p^{n-1}}+\sum_{i=1}^{n-2}\gamma_{i}\del^{p^{i}}+x^{(p^n-p^{n-1})}\sum_{i=0}^{p^{n-1}-1}\lambda_{i}'x^{(i)}\del
\end{align*}
for some $\lambda_{i}'\in k$. If $\gamma_{i}=0$ for all $i$, then (i) of this lemma implies that $\ccD^{p^{n-1}} \in \fL_{(1)}$ which is not conjugate to $\del^{p^{n-1}}$. Similarly, if $j\geq 1$ is the smallest index such that $\gamma_{j}\neq 0$, then (ii)(a) of this lemma implies that $\ccD^{p^{n-1}}\in \fL_{(1)}$ which is again not conjugate to $\del^{p^{n-1}}$. Therefore, the set $\mathcal{S}$ consists of elements of the form $\ccD=\del^{p^{n-1}}+\sum_{i=0}^{n-2}\gamma_{i}\del^{p^{i}}$ with $\gamma_{i}\in k$ such that $\gamma_{0}\neq 0$. This proves statement (ii)(b).
\end{proof}

\begin{cor}\thlabel{c:c2}
Let $\ccD=\partial^{p^{m}}+\sum_{i=0}^{m-1}\alpha_{i}\partial^{p^{i}}+x^{(p^n-p^{m})}\sum_{i=0}^{p^{m}-1}\mu_{i}x^{(i)}\del$ with $1\leq m \leq n-2$ be a nilpotent element of $\fL_p$.
\begin{enumerate}[\upshape(i)]
\item If $\alpha_{i}=0$ for all $i$, then $\ccD^{p^{n-1}}\in \fL_{(1)}$.
\item 
\begin{enumerate}[\upshape(a)]
\item If $q\geq 0$ is the smallest index such that $\alpha_{q}\neq 0$, then $\ccD^{p^{n-1-q}}$ is conjugate under $G$ to 
\[
\del^{p^{n-1}}+x^{(p^{n}-p^{n-1})}\sum_{i=2}^{p^{n-1}-1}\nu_{i}x^{(i)}\del
\]
for some $\nu_{i}\in k$. Hence $\ccD^{p^{n-1}}\in \fL_{(1)}$ for all $q \geq 1$.
\item In particular, if $\alpha_{0}\neq 0$, then $\ccD^{p^{n-1}}$ is conjugate under $G$ to $\del^{p^{n-1}}$. Hence $\ccD=\del^{p^{m}}+\sum_{i=0}^{m-1}\gamma_{i}\del^{p^{i}}$ for some $\gamma_{i}\in k$ with $\gamma_{0}\neq 0$.
\end{enumerate}
\end{enumerate}
\end{cor}
\begin{proof}
Take $\ccD$ as in the corollary. By \thref{c:c1}, one can prove by induction that for any $1\leq r\leq n-1-m$, $\ccD^{p^{r}}$ is conjugate under $G$ to 
\begin{align*}
\partial^{p^{m+r}}+\sum_{i=0}^{m-1}\alpha_{i}^{p^{r}}\del^{p^{i+r}}+x^{(p^n-p^{m+r})}\sum_{i=0}^{p^{m+r}-1}\mu_{i,(r)}x^{(i)}\del
\end{align*}
for some $\mu_{i,(r)} \in k$. In particular, $\ccD^{p^{n-1-m}}$ is conjugate under $G$ to 
\[
\partial^{p^{n-1}}+\sum_{i=0}^{m-1}\alpha_{i}^{p^{n-1-m}}\del^{p^{i+n-1-m}}+x^{(p^n-p^{n-1})}\sum_{i=0}^{p^{n-1}-1}\mu_{i,(n-1-m)}x^{(i)}\del.
\]

(i) If $\alpha_{i}=0$ for all $i$, then \thref{p:p1}(i) implies that $\mu_{0,(n-1-m)}=\mu_{1,(n-1-m)}=0$. By Jacobson's formula, 
\begin{align*}
\ccD^{p^{n-m}}=\sum_{i=2}^{p^{n-1}-1}\mu_{i,(n-1-m)}x^{(i)}\del+\mu_{(n-m)}x^{(p^n-1)}\del 
\end{align*}
for some $\mu_{(n-m)}\in k$. This is an element of $\fL_{(1)}$. Since $G$ preserves $\fL_{(1)}$, this implies that $\ccD^{p^{n-m}}\in \fL_{(1)}$. As $\fL_{(1)}$ is restricted, we have that $\ccD^{p^{n-1}}\in \fL_{(1)}$. This proves statement (i).

(ii) If $q\geq 0$ is the smallest index such that $\alpha_{q}\neq 0$, then $\ccD^{p^{n-1-m}}$ is conjugate under $G$ to
\[
\partial^{p^{n-1}}+\sum_{i=q}^{m-1}\alpha_{i}^{p^{n-1-m}}\del^{p^{i+n-1-m}}+x^{(p^n-p^{n-1})}\sum_{i=0}^{p^{n-1}-1}\mu_{i,(n-1-m)}x^{(i)}\del.
\]
It follows from \thref{p:p1}(ii)(a) that $\ccD^{p^{n-1-q}}$ is conjugate under $G$ to 
\[\del^{p^{n-1}}+x^{(p^n-p^{n-1})}\sum_{i=2}^{p^{n-1}-1}\nu_{i}x^{(i)}\del\]
for some $\nu_{i}\in k$. Suppose now $q\geq 1$, then it is easy to see that $\ccD^{p^{n-q}}$ is an element of $\fL_{(1)}$. Since $\fL_{(1)}$ is restricted, we have that $\ccD^{p^{n-1}}\in \fL_{(1)}$. This proves statement (ii)(a). 

If $\alpha_{0}\neq 0$, then the result follows from above and \thref{p:p1}(ii)(b). This proves statement (ii)(b).
\end{proof}

\subsection{} The calculations in the last subsection enables us to identify an irreducible component of $\cN$.

\begin{lem}\thlabel{l:l14}
Define $\cN_\text{reg}:=\{\ccD\in\cN \,|\,\ccD^{p^{n-1}}\notin \fL_{(0)}\}$. Then 
\[
\cN_\text{reg}=G.(k^*\del+k\del^p+\dots+k\del^{p^{n-1}}).
\]
\end{lem}
\begin{proof}
Since $(\del+\sum_{i=1}^{n-1}\alpha_{i}\del^{p^{i}})^{p^{n}}=0$ and $(\del+\sum_{i=1}^{n-1}\alpha_{i}\del^{p^{i}})^{p^{n-1}}=\del^{p^{n-1}}$, this shows that any elements which are conjugate under $G$ to $\del+\sum_{i=1}^{n-1}\alpha_{i}\del^{p^{i}}$ are contained in $\cN_\text{reg}$. So $G.(k^*\del+k\del^p+\dots+k\del^{p^{n-1}})\subseteq \cN_\text{reg}$. To show that $\cN_\text{reg}\subseteq G.(k^*\del+k\del^p+\dots+k\del^{p^{n-1}})$, we observe that if $\ccD\in \fL_{(1)}$, then $\ccD$ is nilpotent and $\ccD^{p^{n-1}}\in \fL_{(1)}\subset \fL_{(0)}$. Hence $\ccD\not\in \cN_\text{reg}$. Therefore, $\cN_\text{reg}\subseteq \cN\setminus\fL_{(1)}$.

Note that elements of $\cN\setminus\fL_{(1)}$ have the form $\ccD=\sum_{i=0}^{n-1}\alpha_{i}\partial^{{p^{i}}}+f(x)\partial$ for some $f(x)\in \fm$ and $\alpha_i \in k $ with at least one $\alpha_{i}\neq 0$. If $\alpha_{0}\neq 0$ and $\alpha_{i}= 0$ for all $i\geq 1$, then $\ccD=\alpha_{0}\partial+f(x)\partial$. Hence 
\[\ccD^{p^{n-1}}=\alpha_0^{p^{n-1}}\del^{p^{n-1}}+\sum_{i=0}^{n-2}\alpha_i'\del^{p^{i}}+w\]
for some $\alpha_i'\in k$ and $w\in \fL_{(0)}$. As $\alpha_0\neq 0$, this implies that $\ccD^{p^{n-1}} \not \in \fL_{(0)}$ and so $\ccD \in \cN_\text{reg}$. Apply the automorphism $\Phi(x)=\alpha_{0}x$ to $\ccD$, we may assume that $\ccD$ has the form 
\[\partial+g(x)\partial\]
for some $g(x) \in \fm$. By \cite[Lemma~1]{B75}, $\ccD$ is conjugate under $G$ to 
\[\partial+\sum_{i=1}^{n}\beta_{i}x^{(p^{i}-1)}\partial\]
for some $\beta_{i}\in k$. Then it follows from  \cite[Proposition~4.3]{YS16} that $\ccD$ is nilpotent if and only if $\beta_{i}=0$ for all $i$. Consequently, $\ccD$ is conjugate under $G$ to $\partial$. Thus $\cN_\text{reg}\subseteq G.(k^*\del+k\del^p+\dots+k\del^{p^{n-1}})$ in this case. 

For the other elements of $\cN\setminus\fL_{(1)}$, let $1\leq t\leq n-1$ be the largest index such that $\alpha_{t}\neq 0$, i.e. $\ccD=\sum_{i=0}^{t}\alpha_{i}\partial^{{p^{i}}}+f(x)\partial$. Then the automorphism $\Phi(x)=\alpha x$ with $\alpha^{p^{t}}=\alpha_t$ reduces $\ccD$ to the form 
\[\ccD=\del^{p^{t}}+\sum_{i=0}^{t-1}\beta_{i}\partial^{{p^{i}}}+g(x)\del\]
for some $\beta_{i}\in k^{*}\alpha_i$ and $g(x)\in \fm$. It follows from \thref{l:l13} and \thref{c:c1} that $\ccD$ is conjugate under $G$ to 
\[
\del^{p^{t}}+\sum_{i=0}^{t-1}\beta_{i}\partial^{{p^{i}}}+x^{(p^n-p^t)}\sum_{i=0}^{p^t-1}\mu_{i}x^{(i)}\del\]
for some $\mu_{i}\in k$. If $\beta_{i}=0$ for all $i$, then \thref{p:p1} and \thref{c:c2}(i) imply that $\ccD^{p^{n-1}}\in \fL_{(1)}\subset \fL_{(0)}$. If $j\geq 1$ is the smallest index such that $\beta_{j}\neq 0$, then \thref{p:p1} and \thref{c:c2}(ii)(a) imply that $\ccD^{p^{n-1}}\in \fL_{(1)}\subset \fL_{(0)}$. Hence in both cases $\ccD$ is not in $\cN_\text{reg}$. But if $\beta_{0}\neq 0$, then it is easy to see that $\ccD^{p^{n-1}} \not\in \fL_{(0)}$. So $\ccD \in \cN_\text{reg}$. Moreover, it follows from \thref{p:p1} and \thref{c:c2}(ii)(b) that $\ccD$ is conjugate under $G$ to $\del^{p^{t}}+\sum_{i=0}^{t-1}\gamma_{i}\del^{p^{i}}$ for some $\gamma_{i}\in k$ with $\gamma_{0}\neq 0$. Hence $\cN_\text{reg}\subseteq G.(k^*\del+k\del^p+\dots+k\del^{p^{n-1}})$ in this case. Since we have exhausted all elements of $\cN_\text{reg}$, this completes the proof.
\end{proof}
Before we proceed to show that the Zariski closure of $\cN_\text{reg}$ is an irreducible component of $\cN$, we need the following results.
\begin{lem}\thlabel{kerd}
Let $\ccD=\del+\sum_{i=1}^{n-1}\lambda_i\del^{p^{i}}$ with $\lambda_i\in k$ and denote by $\fz_\fL(\ccD)$ $($respectively $\fz_{\fL_{p}}(\ccD))$ the centraliser of $\ccD$ in $\fL$ $($respectively $\fL_p)$. Then
\begin{enumerate}[\upshape(i)]
\item $\fz_\fL(\ccD)=\text{span$\{\del\}$}$.
\item $\fz_{\fL_{p}}(\ccD)=\text{span$\{\del, \del^p, \dots, \del^{p^{n-1}}$\}}$.
\item $\fz_\fL(\ccD)\cap \Lie(G)=\{0\}$.
\end{enumerate}
\end{lem}
\begin{proof}
(i) Clearly, $\text{span$\{\del\}$} \subseteq \fz_\fL(\ccD)$. Since $(\ad\ccD)^{p^{n}-1}\neq 0$ and $(\ad\ccD)^{p^{n}}=0$, the theory of canonical Jordan normal form says that there exists a basis $\mathcal{B}$ of $\fL$ such that the matrix of $\ad\ccD$ with respect to $\mathcal{B}$ is a single Jordan block of size $p^n$ with zeros on the main diagonal. Hence the matrix of $\ad\ccD$ has rank $p^n-1$. This implies that $\ker (\ad\ccD)$ has dimension $1$. By definition, $\ker (\ad\ccD)=\fz_\fL(\ccD)$. Hence $\fz_\fL(\ccD)=\text{span$\{\del\}$}$. 

(ii) It is clear that $\text{span$\{\del, \del^p, \dots, \del^{p^{n-1}}\}$}\subseteq \fz_{\fL_{p}}(\ccD)$. Suppose $v\in \fz_{\fL_p}(\ccD)$. Then we can write $v=\sum_{i=0}^{n-1}\alpha_i\del^{p^{i}}+v_1$ for some $v_1 \in \fL_{(0)}$. Since $\sum_{i=0}^{n-1}\alpha_i\del^{p^{i}}\in \fz_{\fL_p}(\ccD)$, we must have that $v_1 \in \fz_{\fL_p}(\ccD)$. By (i), the centraliser of $\ccD$ in $\fL$ is $k\del$ which is not in $\fL_{(0)}$. Hence $v_1=0$ and $\fz_{\fL_{p}}(\ccD)=\text{span$\{\del, \del^p, \dots, \del^{p^{n-1}}$\}}$.

(iii) It follows from (i) and \thref{LieG}. This completes the proof of the lemma.
\end{proof}

\begin{lem}\thlabel{l:l15}
The Zariski closure of $\cN_\text{reg}$ is an irreducible component of $\cN$.
\end{lem}
\begin{proof}
By \thref{l:l14}, it suffices to show that the Zariski closure of $G.(k^*\del+k\del^p+\dots+k\del^{p^{n-1}})$ is an irreducible component of $\cN$. Put $X:=k^*\del+k\del^p+\dots+k\del^{p^{n-1}}$. Then $X \cong \A^{n-1}$ which is irreducible. Moreover, $G$ is a connected algebraic group so that $\overline{G.X}$ is an irreducible variety contained in $\cN$. Then $\dim \overline{G.X} \leq \dim \cN$. If $\dim \overline{G.X} \geq \dim \cN$, then we get the desired result. 

Define $\Psi$ to be the morphism
\begin{align*}
\Psi: G\times X &\rightarrow \overline{G.X} \\
(g, \ccD)&\mapsto g.\ccD
\end{align*}
Since $G.X$ is dense in $\overline{G.X}$, it contains smooth points of $\overline{G.X}$. As the set of smooth points is $G$-invariant, there exists $\ccD\in X$ such that $\Psi(1, \ccD)=\ccD$ is a smooth point in $\overline{G.X}$. We may assume without loss of generality that $\ccD=\del+\sum_{i=1}^{n-1}\lambda_i\del^{p^{i}}$ for some $\lambda_i\in k$. Then the differential of $\Psi$ at the smooth point $(1, \ccD)$ is the map
\[
d_{(1,\ccD)} \Psi: \Lie(G)\oplus X \rightarrow T_{\ccD}(\overline{G.X}).
\]
Since $\dim T_{\ccD}(\overline{G.X})=\dim \overline{G.X}$, it is enough to show that $\dim T_{\ccD}(\overline{G.X})\geq \dim \cN=p^n-1$. It is easy to see that $T_{\ccD}(\overline{G.X})$ contains $T_{\ccD}(X)=X$ which has dimension $n-1$. Since $\ccD \in X$, $T_{\ccD}(\overline{G.X})$ also contains $T_{\ccD}(G.\ccD)$, the image of $\Lie(G)\oplus \ccD$ under $d_{(1,\ccD)} \Psi$, i.e. 
$T_{\ccD}(G.\ccD)=d_{(1, \ccD)}\Psi(\Lie(G) \oplus \ccD)$. By \thref{kerd}(iii), $\fz_\fL(\ccD)\cap \Lie(G)=\{0\}$, this implies that the restriction of the linear operator $\ad \ccD$ to $\Lie(G)$ has trivial kernel and so the image $[\ccD, \Lie(G)]$ is isomorphic to $\Lie(G)$. Hence 
\[T_{\ccD}(G.\ccD)=d_{(1, \ccD)}\Psi(\Lie(G) \oplus \ccD)=[\ccD, \Lie(G)]\cong \Lie(G).\]
Therefore, $T_{\ccD}(\overline{G.X})$ contains $\Lie(G)$ which has dimension $p^n-n$. It follows from \thref{LieG} that $X \cap \Lie(G)=\{0\}$. Hence $T_{\ccD}(\overline{G.X})$ contains the direct sum $X\oplus \Lie(G)$ of $X$ and $\Lie (G)$. Therefore,
\[
\dim T_{\ccD}(\overline{G.X})\geq \dim \big(X \oplus \Lie(G)\big)=\dim X+\dim \Lie(G)=p^n-1=\dim \cN.
\]
This completes the proof.
\end{proof}

\subsection{}
Our goal is to prove the irreducibility of the variety $\cN$. To achieve this, we need the following result.

\begin{pro}\thlabel{l:l16}
Define $\cN_{\text{sing}}:=\cN\setminus \cN_\text{reg}= \{\ccD \in \cN \,|\, \ccD^{p^{n-1}} \in \fL_{(0)} \}$. Then 
\[
\dim \cN_{\text{sing}}<\dim \cN.
\]
\end{pro}
Clearly, $\cN_{\text{sing}}$ is Zariski closed in $\cN$. To prove this proposition, we need to construct an $(n+1)$-dimensional subspace $V$ in $\fL_{p}$ such that $V \cap \cN_{\text{sing}}=\{0\}$. Then the result follows from \cite[Ch. I, Proposition~7.1]{H77}; see a similar proof in \cite{P91}. The way $V$ is constructed relies on the original definition of $\fL$ due to H.~Zassenhaus and the following lemmas. Recall that $\fL$ has a $k$-basis $\{e_{\alpha}\,|\, \alpha\in \F_q\}$ with the Lie bracket given by $[e_\alpha, e_\beta]=(\beta-\alpha)e_{\alpha+\beta}$. Here $\F_q\subset k$ is a finite field of $q=p^n$ elements. The multiplicative group $\F_q^{\times}$ of $\F_q$ is cyclic of order $p^n-1$ with generator $\xi$; see \S~\ref{2.1} for detail.

\begin{lem}\thlabel{l:l17}
Let $\sigma\in \GL(\fL)$ be such that $\sigma(e_\alpha):=\xi^{-1} e_{\xi\alpha}$ for any $\alpha \in\F_q$. Then $\sigma$ is a diagonalizable automorphism of $\fL$.
\end{lem}
\begin{proof}
By definition, 
\begin{align*}
[\sigma(e_\alpha), \sigma(e_\beta)]=[\xi^{-1} e_{\xi\alpha}, \xi^{-1} e_{\xi\beta}]=\xi^{-2}(\xi\beta-\xi\alpha)e_{\xi\alpha+\xi\beta}=\xi^{-1} (\beta-\alpha)e_{\xi(\alpha+\beta)}=\sigma([e_\alpha, e_\beta])
\end{align*}
for any $\alpha, \beta \in \F_q$. So the endomorphism $\sigma$ is an automorphism of $\fL$. Since $\xi^{p^n-1}=1$, we have that $\sigma^{p^n-1}=\id$. As $k$ is an algebraically closed field, the automorphism $\sigma$ is diagonalizable.
\end{proof}

Since $\sigma$ is an automorphism of $\fL$, it respects the natural filtration $\{\fL_{(i)}\}$ ($i\geq -1$) of $\fL$.
\begin{lem}\thlabel{l:l18}
The automorphism $\sigma$ acts as a scalar on each $1$-dimensional vector space $\fL_{(i)}/\fL_{(i+1)}$.
\end{lem}
\begin{proof}
Let $T$ denote the torus of the $p$-envelope $\langle e_{0}\rangle_{p}$ in $\fL_p$ generated by $e_{0}$. Let $\F_p\subset k$ denote the finite field with $p$ elements. By \cite[Theorem~1.3.11(1)]{S04}, $\dim_k T$ is the $\F_p$-dimension of the $\F_p$-vector space spanned by the $T$-roots; see also the proof of \cite[Theorem~7.6.3(2)]{S04}. Since $[e_0, e_\beta]=\beta e_\beta$ for any $\beta\in \F_q$, the endomorphism $\ad(e_{0})$ has $q=p^n$ distinct eigenvalues. Therefore, $\dim_k T=n$. As $\sigma(e_0)=\xi^{-1}e_0$ and $\sigma(e_0^{p^{j}})=\xi^{-p^{j}}e_0^{p^{j}}$ for all $j\geq 1$, we see that $\xi^{-1}, \xi^{-p}, \xi^{-p^2},\dots, \xi^{-p^{n-1}}$ are the eigenvalues of $\sigma$ on $T$. Note that $e_{0}\notin \fL_{(0)}$. Indeed, if $e_0\in \fL_{(0)}$, then $T$ is contained in $\fL_{(0)}$ as $\fL_{(0)}$ is restricted. But this contradicts the fact that any torus of $\fL_{(0)}$ has dimension $1$ \cite[p. 67]{T98}. Therefore, $e_{0}\notin \fL_{(0)}$. As a result, $T\cap \fL=ke_0$.

Consider the surjective map $\pi: \fL_{(-1)} \twoheadrightarrow \fL_{(-1)}/\fL_{(0)}$. Since $e_{0}\notin \fL_{(0)}$, the vector space $\fL_{(-1)}/\fL_{(0)}=k\del$ is spanned by $\pi(e_0)$. This implies that $e_0$ has weight $-1$ and $\sigma$ acts on $\fL_{(-1)}/\fL_{(0)}$ as $\xi^{-1}\id$. Similarly, we can show that $\sigma$ acts on $\fL_{(k)}/\fL_{(k+1)}$ as $\xi^{k}\id$ for $0\leq k\leq p^n-2$. Indeed, elements of $\fL_{(k)}/\fL_{(k+1)}$ have the form $x+\fL_{(k+1)}$ for some $x\in \fL_{(k)}$. Since $e_0\in \fL_{(-1)}\setminus \fL_{(0)}$, we have that $[e_0, \fL_{(k+1)}]\subseteq [\fL_{(-1)}, \fL_{(k+1)}]\subseteq\fL_{(k)}$. Hence $[e_0, \fL_{(k+1)}]+\fL_{(k+1)}=\fL_{(k)}$. In particular, $[e_0, \fL_{(0)}]+\fL_{(0)}=\fL_{(-1)}$. If $u\in \fL_{(0)}$ is such that $[e_0, u]\notin \fL_{(0)}$, i.e. $[e_0, u]\neq 0$ on $\fL_{(-1)}/\fL_{(0)}$, then $u$ is an eigenvector of $\sigma$ corresponding to an eigenvalue, say $\lambda$. Then
\begin{align*}
\sigma[e_0, u]=[\sigma(e_0), \sigma(u)]=[\xi^{-1}e_0, \lambda u]=\xi^{-1}\lambda[e_0, u].
\end{align*}
So $\xi^{-1}\lambda$ is the eigenvalue of $\sigma$ on $\fL_{(-1)}/\fL_{(0)}$. But $\sigma$ acts on $\fL_{(-1)}/\fL_{(0)}$ as $\xi^{-1}\id$, we must have that $\xi^{-1}\lambda=\xi^{-1}\id$. Thus $\lambda =1$, i.e. $\sigma$ acts on $\fL_{(0)}/\fL_{(1)}$ as $\id$. Continue in this way, one can show that $\sigma$ acts on $\fL_{(1)}/\fL_{(2)}$ as $\xi \id$, on $\fL_{(k)}/\fL_{(k+1)}$ as $\xi^k\id$ and on $\fL_{(p^n-2)}$ as $\xi^{p^n-2}\id=\xi^{-1}\id$. This completes the proof.
\end{proof}

\begin{rmk}\thlabel{rkV}
The last lemma shows that
\begin{enumerate}[\upshape(i)]
\item the eigenvalues of $\sigma$ on $\fL$ are $\xi^{-1}, \xi^0=1, \xi, \dots, \xi^{p^{n}-3}$ and $\xi^{p^{n}-2}=\xi^{-1}$. All have multiplicity $1$ except $\xi^{-1}$ which has multiplicity $2$; 
\item the eigenvalues of $\sigma$ on $\fL_{(0)}$ are $\xi^0=1, \xi, \dots, \xi^{p^{n}-3}$ and $\xi^{p^{n}-2}=\xi^{-1}$. All have multiplicity $1$;
\item the eigenspace $\fL[k]:=\{\ccD\in \fL\,|\,\sigma(\ccD)=\xi^{k}\ccD\}$ corresponding to the eigenvalue $\xi^{k}$, where $0\leq k\leq p^{n}-3$, has dimension $1$. In particular, the eigenspace $\fL[0]=ku$, which is a torus in $\fL_{(0)}$. Since any torus has a basis consisting of toral elements, we may assume that $u$ is toral, i.e. $u^p=u$;
\item the eigenspace $\fL[-1]=\text{span $\{e_0, v\,| \,v\in \fL_{(p^{n}-2)}\}$}$ and has dimension $2$.
\end{enumerate}
\end{rmk}

\begin{proof}[Proof of \thref{l:l16}] 
Recall that the $n$-dimensional torus  $T=\langle e_0\rangle_p$ in $\fL_p$ and the toral element $u\in \fL_{(0)}\setminus \fL_{(1)}$; see \thref{l:l18} and \thref{rkV}(iii). Put $V:=T\oplus ku=\sum_{i=0}^{n-1}ke_0^{p^{i}}\oplus ku$. We want to show that $V\cap \cN_{\text{sing}}=\{0\}$.

Suppose for contradiction that $V\cap \cN_{\text{sing}}\neq\{0\}$. Then take a nonzero element $y$ in $V\cap \cN_{\text{sing}}$, we can write
\begin{align*}
y=\sum_{i=0}^{n-1}\lambda_{i}e_{0}^{p^{i}}+\mu u
\end{align*}
for some $\lambda_i, \mu\in k$ with at least one $\lambda_i\neq 0$. Suppose $\lambda_0\neq 0$. Since $e_{0}\in \fL_{(-1)}\setminus\fL_{(0)}$, we may assume without loss of generality that $e_0=\del+\sum_{i=1}^{p^{n}-1}\alpha_{i}x^{(i)}\del$ for some $\alpha_{i}\in k$. As the $p$-th power map on $T$ is periodic, i.e. $(e_0^{p^{k}})^{p^{n-1}}=e_0^{p^{k-1}}$ for all $k\geq 1$, this implies that 
\[y^{p^{n-1}}=\lambda_0^{p^{n-1}}\del^{p^{n-1}}+\sum_{i=0}^{n-2}\lambda_i'\del^{p^{i}}+w
\]
for some $\lambda_i'\in k$ and $w\in \fL_{(0)}$. Since $\lambda_0\neq 0$, this shows that $y^{p^{n-1}}$ is not in $\fL_{(0)}$, a contradiction.

Suppose now $\lambda_0=0$ and let $1\leq s\leq n-1$ be the largest index such that $\lambda_s\neq 0$. Then 
\[
y=\sum_{i=1}^{s}\lambda_{i}e_{0}^{p^{i}}+\mu u. 
\]
By \cite[Lemma~1.1.1]{S04} and the fact that $e_0^{p^{n}}=e_0$, we have that 
\begin{align}\label{ypn-s}
y^{p^{n-s}}=\lambda_{s}^{p^{n-s}}e_{0}+\lambda_{s-1}^{p^{n-s}}e_{0}^{p^{n-1}}+\dots+\lambda_{1}^{p^{n-s}}e_{0}^{p^{n-s+1}}+\mu^{p^{n-s}}u+\sum_{l=0}^{n-s-1}v_l^{p^{l}},
\end{align}
where $v_l$ is a linear combination of commutators in $e_0^{p^{j}}$ ($1\leq j\leq s$) and $u$. By the Jacobi identity, we can rearrange each $v_l$ so that
\begin{align*}
v_l \in \text{span$\{[e_0^{a_0}[u^{b_1}[e_0^{a_1}[u^{b_2}[e_0^{a_2}[\dots[e_0^{a_{t-1}}[u^{b_t}[e_0^{a_t}, u]\dots]\}$}, 
\end{align*}
where $[e_0^{a_0}[u^{b_1}[e_0^{a_1}[u^{b_2}[e_0^{a_2}[\dots[e_0^{a_{t-1}}[u^{b_t}[e_0^{a_t}, u]\dots]$ is a left normed commutator of length $p^{n-s-l}$ with $u$ at the right end, and $p\leq a_{i}\leq p^{s}, b_i$ are arbitrary constants. Since $e_{0}$ and $u$ are eigenvectors of $\sigma$ corresponding to eigenvalues $\xi^{-1}$ and $\xi^0=1$, respectively, the commutator $[e_0^{a_0}[u^{b_1}[e_0^{a_1}[u^{b_2}[e_0^{a_2}[\dots[e_0^{a_{t-1}}[u^{b_t}[e_0^{a_t}, u]\dots]$ is an eigenvector of $\sigma$ corresponding to the eigenvalue $\xi^{-(a_0+a_1+\dots+a_t)}$. As 
\[
a_0+a_1+\dots+a_t\leq (p^{n-s-l}-1)p^s\leq (p^{n-s}-1)p^s=p^n-p^s\leq p^n-p\neq 1,
\]
the eigenvalue $\xi^{-(a_0+a_1+\dots+a_t)}$ is not equal to $\xi^{-1}$. Hence $v_l \in \fL_{(0)}\setminus \fL_{(p^n-2)}\subset \fL_{(0)}$. Since $\fL_{(0)}$ is restricted, we have that $v_l^{p^{l}}\in \fL_{(0)}$ and so $\sum_{l=0}^{n-s-1}v_l^{p^{l}}\in \fL_{(0)}$. As $e_0$ is not in $\fL_{(0)}$, this shows that the term $\sum_{l=0}^{n-s-1}v_l^{p^{l}}$ does not cancel with the first term $\lambda_{s}^{p^{n-s}}e_{0}$ in $y^{p^{n-s}}$ \eqref{ypn-s}. Therefore, 
\begin{align*}
y^{p^{n-s}}=\lambda_{s}^{p^{n-s}}e_{0}+\lambda'_{n-1}e_{0}^{p^{n-1}}+\dots+\lambda'_{n-s+1}e_{0}^{p^{n-s+1}}+w_1
\end{align*}
for some $\lambda'_{i}\in k$ and $w_1\in \fL_{(0)}$. Then we know that $(y^{p^{n-s}})^{p^{n-1}}=y^{p^{2n-s-1}}$ belongs to $\lambda_{s}^{p^{2n-s-1}}e_{0}^{p^{n-1}}+\sum_{i=0}^{n-2}\fL^{p^{i}}$. As $\lambda_s\neq 0$, this implies that $y^{p^{2n-s-1}}$ is not equal to $0$. But $2n-s-1\geq n$, this contradicts that $y$ is nilpotent. Therefore, we proved by contradiction that $V\cap \cN_{\text{sing}}=\{0\}$. The result then follows from \cite[Ch. I, Proposition~7.1]{H77}. This completes the proof.
\end{proof}

\begin{thm}
The variety $\cN$ is irreducible.
\end{thm}
\begin{proof}
The variety $\cN$ is equidimensional of dimension $p^n-1$. The ideal defining $\cN$ is homogeneous, hence any irreducible component of $\cN$ contains $0$ \cite[Theorem 4.2]{P03}.
It follows from \thref{l:l15} that the Zariski closure of $\cN_\text{reg}$ is an irreducible component of $\cN$. Let $Z_1, \dots ,Z_t$ be pairwise distinct irreducible components of $\cN$, and set $Z_1=\overline{\cN_\text{reg}}$. Suppose $t\geq 2$. Then $Z_2\setminus Z_1$ is contained in $\cN_{\text{sing}}$, which is Zariski closed in $\cN$ with $\dim \cN_{\text{sing}}<\dim \cN$  by \thref{l:l16}. Since $Z_2\setminus Z_1=Z_2 \setminus (Z_1\cap Z_2)$, this set is Zariski dense in $Z_2$. Then its closure $Z_2$ is also contained in $\cN_{\text{sing}}$, i.e. $\dim Z_2=\dim \cN \leq \dim \cN_{\text{sing}}$. This is a contradiction. Hence $t=1$ and the variety $\cN$ is irreducible. This completes the proof of \thref{main}.
\end{proof}

\section*{Acknowledgements}
I would like to thank Professor Alexander Premet for his guidance and encouragement during my research.
I am very grateful to him for having read this paper carefully and suggested several improvements. I would also like to acknowledge the support of a school scholarship from the University of Manchester.


\begin{thebibliography}{99}
\bibitem{B75}
G.~Brown, \emph{Cartan subalgebras of Zassenhaus algebras}, Canad. J. Math. 27 (1975), no.~5, 1011--1021. 

\bibitem{H77}
R.~Hartshorne, \emph{Algebraic Geometry}, Graduate Texts in Mathematics, no.~52. Springer-Verlag, New York-Heidelberg, 1977.

\bibitem{J04}
J.~Jantzen,
\emph{Nilpotent Orbits in Representation Theory}, in: J.-P.~Anker, B.~Orsted (eds), Lie Theory: Lie Algebras and Representations, Progress in Mathematics, 228, Birkh{\"a}user Boston, Boston, MA, 2004, 1--211.

\bibitem{P87}
A.~Premet,
\emph{On Cartan subalgebras of Lie $p$-algebras}, Izv. Acad. Nauk SSSR, Ser. Mat. 50 (1986), no.~4, 788--800 (Russian); Math. USSR Izv. 29 (1987), no.~1, 145--157 (English translation).

\bibitem{P90}
A.~Premet,
\emph{Regular Cartan subalgebras and nilpotent elements in restricted Lie algebras}, Mat.
Sbornik 180 (1989), no.~4, 542--557 (Russian); Math. USSR Sbornik 66 (1990), no.~2, 555--570
(English translation).

\bibitem{P91}
A.~Premet,
\emph{The theorem on restriction of invariants, and nilpotent elements in $W_{n}$}, Mat.
Sbornik 182 (1991), no.~5, 746--773 (Russian); Math. USSR Sbornik 73 (1992), no.~1, 135--159
(English translation).

\bibitem{P03}
A.~Premet,
\emph{Nilpotent commuting varieties of reductive Lie algebras}, Invent. Math. 154 (2003),
653--683.

\bibitem{R56}
R.~Ree,
\emph{On generalized Witt algebras}, Trans. Amer. Math. Soc. 83 (1956), 510--516.

\bibitem{YS16}
Y.-F.~Yao, B.~Shu, \emph{Nilpotent orbits of certain simple Lie algebras over truncated polynomial rings}, J. Algebra 458 (2016), 1–-20. 

\bibitem{S88}
H.~Strade and R.~Farnsteiner, \emph{Modular Lie Algebras and their Representations}, Monographs and Textbooks in Pure and Applied Mathematics, 116, Marcel Dekker, Inc., New York, 1988. 

\bibitem{S04}
H.~Strade, \emph{Simple Lie Algebras Over Fields of Positive Characteristic, Vol. 1. Structure Theory}, 2\textsuperscript{nd} edition, De Gruyter Expositions in Mathematics, 38, De Gruyter, Berlin, 2017.

\bibitem{T98}
S.~A.~Tyurin, \emph{Classification of tori in the Zassenhaus algebra}, Izv. Vyssh. Uchebn. Zaved. Mat. (1998), no.~2, 69--76 (Russian); Russian Math. (Iz. VUZ) 42 (1998), no.~2, 66–-73 (English translation). 

\bibitem{WCL13}
J.~Wei, H.~Chang and X.~Lu,
\emph{The variety of nilpotent elements and invariant polynomial functions on the special algebra $S_{n}$}, Forum Math. 27 (2015), no.~3, 1689–-1715.

\bibitem{W14}
J.~Wei,
\emph{The nilpotent variety and invariant polynomial functions in the Hamiltonian algebra}, arXiv:1401.6532v1 [math.RT], 2014.

\bibitem{W71}
R.~Wilson,
\emph{Classification of generalized Witt algebras over algebraically closed fields}, Trans. Amer. Math. Soc. 153 (1971), 191–-210. 
\end{thebibliography}
\end{document}